\documentclass[12pt, a4paper]{amsart}
\usepackage[margin=1.1 in]{geometry}
\usepackage{amssymb}
\usepackage{csquotes}
\usepackage{enumitem}
\usepackage{xcolor}

\theoremstyle{plain}
\newtheorem{theorem}{Theorem}[section]
\newtheorem{lemma}[theorem]{Lemma}
\newtheorem{proposition}[theorem]{Proposition}
\newtheorem{corollary}[theorem]{Corollary}

\theoremstyle{definition}
\newtheorem{definition}[theorem]{Definition}

\newtheorem{remark}[theorem]{Remark}
\newtheorem{example}[theorem]{Example}

\newenvironment{claim 1}[1]
{\par%\noindent
\underline{Claim 1:}\space#1}{}
\newenvironment{claim 2}[1]
{\par%\noindent
\underline{Claim 2:}\space#1}{}
\newenvironment{claim 3}[1]
{\par%\noindent
\underline{Claim 3:}\space#1}{}
\newenvironment{claim}[1]
{\par%\noindent
\underline{Claim:}\space#1}{}

\newenvironment{claimproof}[1]{\par\noindent\underline{Proof:}\space#1}{\hfill $\blacksquare$}

\numberwithin{equation}{section}

\title[Surjectivity of \(\overline{\partial}\) and \(\Delta\) on weighted spaces of  \(C^{\infty}\)-functions ]{On the surjectivity of the Cauchy-Riemann and Laplace operators on weighted spaces of smooth functions}

\author[A. Debrouwere]{Andreas Debrouwere}
\address{A. Debrouwere, Department of Mathematics and Data Science \\ Vrije Universiteit Brussel, Belgium\\ Pleinlaan 2 \\ 1050 Brussels \\ Belgium}
\email{Andreas.Debrouwere@vub.be}

\author[Q. Van Boxstael]{Quinten Van Boxstael}
\address{Q. Van Boxstael,  Department of Mathematics and Data Science \\ Vrije Universiteit Brussel, Belgium\\ Pleinlaan 2 \\ 1050 Brussels \\ Belgium}
\email{Quinten.Van.Boxstael@vub.be}

\author[J. Vindas]{Jasson Vindas}
\thanks {J. Vindas acknowledges support by the Ghent University grant  bof/baf/4y/2024/01/155 and the Research Foundation–Flanders grant 
G067621N}
\address{J. Vindas, Department of Mathematics: Analysis, Logic and Discrete Mathematics\\ Ghent University\\ Krijgslaan 281\\ 9000 Ghent\\ Belgium}
\email{Jasson.Vindas@UGent.be}

\subjclass[2020]{\emph{Primary.} 35A01, 46E10.  \emph{Secondary.} 35A35, 35J05. }

\keywords{Surjectivity of the Cauchy-Riemann and Laplace operator; weighted spaces of smooth functions; weighted Runge type approximation theorem}

\begin{document}
\begin{abstract}
    We study the surjectivity of the Cauchy-Riemann and Laplace operators on certain  weighted spaces of smooth functions of rapid decay on strip-like domains in the complex plane that are defined via 
    weight function systems. We fully characterize when these operators are surjective on such function spaces in terms of a growth condition on the defining  weight function systems.
    \end{abstract}
\maketitle
\section{Introduction}
Characterizing when a constant coefficient partial differential operator (PDO) on $\mathbb{R}^d$ is surjective on a given space
of functions or distributions is a fundamental problem in functional analysis that goes back to the pioneering
works of Ehrenpreis \cite{Ehrenpreis}, Malgrange \cite{Malgrange}, and H\"ormander \cite{Hormander-intro}. This question has been extensively studied for local spaces, e.g.\ the space $C^\infty(X)$ of smooth functions, the space $\mathcal{D}'(X)$ of distributions, or the spaces $B^{\operatorname{loc}}_{p,k}(X)$, where $X \subseteq \mathbb{R}^d$ is open; see H\"ormander's monograph \cite{HoermanderPDO2}.

The surjectivity problem is however much less understood for global weighted spaces.  
Let us mention a few results in this direction.  A classical result, independently shown by H\"ormander \cite{Hormander-div} and {\L}ojasiewicz  \cite{Loja}, states that every non-zero  PDO is surjective on the space $\mathcal{S}'(\mathbb{R}^d)$ of tempered distributions. In \cite{Larcher,O-W}, the surjectivity of PDO on the space $\mathcal{O}_M(\mathbb{R}^d)$  of slowly increasing smooth functions was studied. The related question concerning the existence of continuous linear right inverses of PDO (and, more generally, convolution operators) on various weighted function and (ultra)distribution spaces has been thoroughly investigated by Langenbruch, see \cite{Langenbruch} and references therein.

In this paper, we study the surjectivity of the Cauchy-Riemann operator \(\displaystyle \overline{\partial} = \frac{1}{2}\left(\frac{\partial}{\partial x} + i \frac{\partial}{\partial y}\right)\)  and the Laplace operator  \(\displaystyle \Delta = \frac{\partial^2}{\partial x^2} + \frac{\partial^2}{\partial y^2}\) on 
 weighted spaces of smooth functions 
 on certain open subsets of 
 $\mathbb{C}$. More precisely, we  introduce the weighted spaces $\mathcal{K}_W(X)$ of smooth functions of rapid decay on $X$, where  $X \subseteq \mathbb{C}$ is a generalized strip (which can  more or less be seen as a generalization of a horizontal strip of the form $\mathbb{R} + i(-h,h)$, $h >0$, but that may have a uniformly continuous curve as boundary) 
  and $W$ is a weight function system (measuring how fast the functions in $\mathcal{K}_W(X)$ decay at infinity); see Section \ref{Prelimaries} for the precise definition of $\mathcal{K}_W(X)$.  Our goal is then to characterize when \(\displaystyle \overline{\partial}\) and $\Delta$ are surjective on $\mathcal{K}_W(X)$ in terms of a growth condition on the defining weight function system $W$. Our results complement and extend the recent work of Kruse \cite{Kruse1} on this problem for the Cauchy-Riemann operator (see also the related papers \cite{Kruse2,Kruse3}).

We now state a sample of our main result. We define the horizontal strip \(T_h = \mathbb{R} + i(-h,h)\) for $h \in (0,\infty]$. \begin{theorem}\label{theorem 1}\label{Theorem intro CR and L}
    Let \(w \colon [0,\infty) \rightarrow  [0,\infty)\) be a non-decreasing unbounded function. Suppose that there is \(C>1\) such that  
$$
\int_0^{\infty} e^{w(t)-w(Ct)}\text{d}t < \infty.
$$
Let $h\in (0,\infty].$ 
 Define \(\mathcal{K}_{(w)}(T_h)\) as the space consisting of all \(f \in C^{\infty}(T_h)\) such that 
\[\sup_{z\in \overline{T_{h^{\prime}}}}e^{w(N\lvert \operatorname{Re  }z\rvert)} \lvert f^{(\alpha)}(z)\rvert < \infty, \qquad \forall N \in  \mathbb{N}, \alpha \in \mathbb{N}^2, h' \in (0,h).\]
Then, the following statements are equivalent:
    \begin{enumerate}[label=(\roman*)]
        \item The Cauchy-Riemann operator \(\overline{\partial} \colon \mathcal{K}_{(w)}(T_h) \rightarrow \mathcal{K}_{(w)}(T_h)\) is surjective.
        \item The Laplace operator \(\Delta \colon \mathcal{K}_{(w)}(T_h) \rightarrow \mathcal{K}_{(w)}(T_h)\) is surjective.
        \item There exists a holomorphic function in \(\mathcal{K}_{(w)}(T_h)\) that is not identically zero.
        \item For all $\mu > 0$,        
        $$
        \displaystyle \int_0^{\infty} w(t) e^{-\mu t}\mathrm{d}t < \infty.
        $$
    \end{enumerate}
\end{theorem}
The main improvement of Theorem \ref{theorem 1} upon the  work of Kruse \cite{Kruse1} is that the conditions on the weight function $w$ are much less restrictive than in  \cite{Kruse1}, leading to a new full characterization of the surjectivity of \(\overline{\partial}\) and \(\Delta\) on  $\mathcal{K}_{(w)}(T_h)$. As a concrete example, we mention that  Theorem \ref{theorem 1}  implies that  \(\overline{\partial}\) and \(\Delta\) are surjective on  $\mathcal{K}_{(w)}(T_h)$ for $w(t) =  t^a$ with \(a > 0\). This was previously shown in  \cite[Example 5.7]{Kruse1} for the Cauchy-Riemann operator, but under the additional assumption 
 \(a  \leq 1\). On the other hand, 
 it is worth pointing out that \cite[Example 5.7]{Kruse1} holds for spaces defined on a broader 
  class of open sets than (generalized) strips. 

The principal tool in the proof of Theorem \ref{theorem 1}  (and, more generally, in that of our main result Theorem \ref{main theorem}) is a novel weighted version of the classical Runge approximation theorem, which we also show here and  may be of independent interest. This makes our proof method elementary and constructive. Furthermore, it is  different from the approach used in \cite{Kruse1}, which is based on the Hahn-Banach theorem and is inspired by H\"ormander's solution  
 of the \(\overline{\partial}\)-problem in weighted $L^2$-spaces.

The outline of this paper is as follows. In the preliminary Section \ref{Prelimaries}, we  introduce the weighted spaces of smooth functions that  we are interested in and recall the abstract Mittag-Leffler lemma for Fr\'echet spaces. Next, in Section 
\ref{main section}, we discuss our main result (Theorem \ref{main theorem}) and give various examples illustrating it.  The main ingredient in the proof of Theorem \ref{main theorem}, our new Runge type approximation result for weighted spaces of smooth functions, is shown in Section \ref{Runge type result}.
Finally, in Section \ref{proof main theorem},  we prove 
 Theorem \ref{main theorem}.
\section{Preliminaries}\label{Prelimaries}
In this preliminary section, we introduce generalized strips, weight function systems, and the weighted spaces  of smooth functions of rapid decay  that we shall be concerned with. We also explain the abstract Mittag-Leffler lemma for Fr\'echet spaces.

\subsection{Generalized strips}
The following special kind of open sets in $\mathbb C$ will play a fundamental role in this article. 
\begin{definition}
We write  \(\mathcal{F}(\mathbb{R})\)  for the family of all functions $F: \mathbb{R} \to (0,\infty)$ that are uniformly continuous and satisfy   \(0< \inf_{t\in \mathbb{R}} F(t) \leq \sup_{t\in \mathbb{R}} F(t)<\infty \). 
Given \(F, G \in \mathcal{F}(\mathbb{R})\), we define the \emph{generalized strip} \(T^{F,G}\) as
 \[T^{F,G}= \{z \in \mathbb{C} \, \vert \, -G(\mathrm{Re \ }z) < \mathrm{Im \ }z < F(\mathrm{Re \ }z)\} .\]
\end{definition}

We define  the horizontal strip \(T_h = \mathbb{R} + i(-h,h)\) for $h \in (0,\infty)$. The following two lemmas will be used later on. They show that, in a certain sense, generalized strips are well-separated in distance and by the graphs of two \(C^{\infty}\)-functions.
\begin{lemma}\label{Lemma 1}
    Let \(F, G \in \mathcal{F}(\mathbb{R})\) and \(a \in (0,1)\). Then, there is \(\varepsilon > 0\) such that 
    $$
    \overline{T^{a F, a G}} + \overline{B}(0, \varepsilon) \subseteq T^{F,G}.
    $$
\end{lemma}
\begin{proof}
We just need to prove that there is \(\varepsilon > 0\) such that \(\{x+iaF(x)\mid x \in \mathbb{R})\} + \overline{B}(0,\varepsilon) \subseteq T^{F,G}\) and \(\{x-iaG(x)\mid x \in \mathbb{R}\} + \overline{B}(0,\varepsilon) \subseteq T^{F,G}\). We only show the former  inclusion as the latter
 can be proved in a completely analogous way. 
 Choose \( \displaystyle0 < \varepsilon^{\prime} < \frac{1-a}{1+a} \inf_{t \in \mathbb{R}} F(t)\). Since \(F\) is uniformly continuous on \(\mathbb{R}\), there is \(\delta > 0\) such that 
\begin{equation}
    \forall t,t^{\prime} \in \mathbb{R}: \lvert t - t^{\prime}\rvert \leq \delta \implies \lvert F(t)-F(t^{\prime})\rvert \leq \varepsilon^{\prime}.
\end{equation}
Pick \(0<\varepsilon < \min \{\delta, \varepsilon^{\prime}, a \inf_{t\in \mathbb{R}} F(t)\}\). Let $x \in \mathbb{R}$ be arbitrary. Let \(x^{\prime}+iy^{\prime} \in \overline{B}(x+iaF(x), \varepsilon)\). 
It is clear that \(-G(x^{\prime})<0<y^{\prime}\) and
\begin{align*}
    y^{\prime} &\leq \lvert y^{\prime}-aF(x)\rvert + aF(x) < \varepsilon^{\prime} + a\lvert F(x)-F(x^{\prime})\rvert + a F(x^{\prime}) \\
    &\leq \varepsilon^{\prime}(1+a) + a F(x^{\prime}) \leq (1-a)F(x^{\prime}) + aF(x^{\prime}) = F(x^{\prime}).
\end{align*}
Hence, \(\{x+iaF(x)\vert x \in \mathbb{R}\} + \overline{B}(0,\varepsilon) \subseteq T^{F,G}\).
\end{proof}
\begin{lemma}\label{smooth interpolation}
Let \(F,G \in \mathcal{F}(\mathbb{R})\) and \(a \in (0,1)\). Then, there are \(\varepsilon > 0\) and \(\Phi, \Psi \in \mathcal{F}(\mathbb{R})\cap C^{\infty}(\mathbb{R})\), with all derivatives bounded, 
   such that \(\overline{T^{aF,aG}} + \overline{B}(0,\varepsilon) \subseteq T^{\Phi, \Psi}\) and \(\overline{T^{\Phi,\Psi}} + \overline{B}(0,\varepsilon) \subseteq T^{F,G}\).   
\end{lemma}
\begin{proof}
Take \(c \in (a,1)\) arbitrary. By Lemma \ref{Lemma 1}, we can choose \(\varepsilon > 0\) such that  \(\overline{T^{aF,aG}} + \overline{B}(0,2\varepsilon) \subseteq T^{cF,cG}\) and \(\overline{T^{cF,cG}} + \overline{B}(0,2\varepsilon) \subseteq T^{F,G} \label{strip plus ball}\). 
Since \(F\) and \(G\) are uniformly continuous on $\mathbb{R}$, there is \(\delta > 0\) such that, for all $t,t' \in \mathbb{R}$,  \[\lvert t-t^{\prime}\rvert \leq \delta \implies \lvert c F(t)-c F(t^{\prime})\rvert \leq \varepsilon \text{ and } \lvert c G(t) -c G(t^{\prime})\rvert \leq \varepsilon .\] Choose now \(\varphi \in C^{\infty}(\mathbb{R}^2)\) with \(\operatorname{ supp } \varphi \subseteq B(0,\delta)\), \(\varphi \geq 0\), \(\|\varphi\|_{L^1(\mathbb{R}^2)} = 1\), and define \(\Phi = cF*\varphi\), \(\Psi = cG*\varphi\). It is clear that \(\Phi, \Psi \in C^{\infty}(\mathbb{R})\) and that all derivatives of \(\Phi, \Psi\) are bounded. By the mean value theorem, the latter implies that \(\Phi\) and \(\Psi\) are uniformly continuous. Using
 the estimates \(\|cF-\Phi\|_{L^{\infty}(\mathbb{R})} \leq\)  \(\sup_{x\in\mathbb{R}}\| cF(x) \varphi - cF(x-\cdot)\varphi\|_{L^1(\overline{B}(0,\delta))} \leq \varepsilon\)
 and \(\|cG-\Psi\|_{L^{\infty}(\mathbb{R})} \leq \varepsilon\), 
 it follows that \(\Phi, \Psi \in \mathcal{F}(\mathbb{R})\), \(\overline{T^{aF,aG}} + \overline{B}(0,\varepsilon) \subseteq T^{\Phi,\Psi}\), and \(\overline{T^{\Phi,\Psi}} + \overline{B}(0,\varepsilon) \subseteq T^{F,G}\). 

\end{proof}

\subsection{Weight function systems}
We will measure the decay of functions at infinity via so called weight function systems, a concept that is introduced in the following definition.
\begin{definition}
 A function \(w \colon [0,\infty) \rightarrow  [0,\infty)\) is called a \emph{weight function} if it is non-decreasing and unbounded. A pointwise non-decreasing sequence  \( W = (w_N)_{N \in \mathbb{N}}\)  of weight functions is called a \emph{weight function system}.  
\end{definition}

We shall employ the following conditions on a weight function system  \( W = (w_N)_{N \in \mathbb{N}}\):
\begin{enumerate}
    \item [\((\alpha)\)] \(\forall N \in \mathbb{N} \, \exists M > N, A > 1 \, \forall t\geq 0 : w_N(2t) \leq w_M(t) + \log A\).  
      \\
    \item [\((\epsilon)_0\)] \(\forall N \in \mathbb{N} \, \forall \mu > 0: \displaystyle{\int_0^{\infty} w_N(t) e^{-\mu t}\mathrm{d}t < \infty}\).
  \item [ \((N)\)] \(\forall N \in \mathbb{N}\, \exists M > N: \displaystyle{\int_{0}^{\infty} e^{w_N(t) - w_M(t)} \mathrm{d}t< \infty}\).
  \end{enumerate}
  
The conditions $(\alpha)$ and $(N)$ are of a technical nature, while \((\epsilon)_0\) shall play a fundamental role  in our considerations (see Theorem \ref{main theorem} 
below).
Condition $(\alpha)$ implies that $W$ is weakly subadditive in the following sense.

\begin{lemma}
\label{Subadditivity}
 Let \(W = (w_N)_{N \in \mathbb{N}}\) be a weight function system that satisfies \((\alpha)\). Then,
 \[
 \forall N \in \mathbb{N} \, \exists M > N, A > 1 \, \forall t,s \geq 0 : w_N(t+s) \leq w_M(t) + w_M(s) + \log A. 
 \]
\end{lemma}
\begin{proof}
    Take \(N \in \mathbb{N}\) arbitrary and choose \(M > N\) and  \(A > 1\) as in \((\alpha)\). Then, for all \(t,s \geq 0\),
    \[
    w_N(t+s) \leq w_N(2t) + w_N(2s) \leq w_M(t) + w_M(s) + 2\log A.
    \]
\end{proof}

Given $X  \subseteq \mathbb{C}$ open, we write $\mathcal{H}(X)$ for the space of holomorphic functions on \(X\). The following lemma shall be of crucial importance to us. 
\begin{lemma}\label{Lemma 2}
    Let \(W = (w_N)_{N \in \mathbb{N}}\) be a weight function system that satisfies \((\epsilon)_0\). Then, for all \(N \in \mathbb{N}\) and \(h > 0\), there is \(Q \in \mathcal{H}(T_h)\) such that 
    \[
    e^{w_N(\lvert \operatorname{Re}\xi\rvert)} \leq \lvert Q (\xi)\rvert, \qquad \forall \xi \in T_h.
    \]
\end{lemma}
\begin{proof}
This follows immediately from \cite[Proposition 3.2]{debrouwere2018non}.
  \end{proof}
\subsection{Weighted spaces of smooth and holomorphic functions} 
We are ready to define the weighted spaces of smooth and holomorphic functions that we are interested in.
\begin{definition}
Let \(W =  (w_N)_{N \in \mathbb{N}}\) be a weight function system and let \(F, G \in \mathcal{F}(\mathbb{R})\). We define $\mathcal{K}_{W}(T^{F,G})$ as the space consisting of all 
$f \in C^{\infty}(T^{F,G})$ such that 
$$
p_{N,\alpha,a}(f) =  \sup_{\xi \in \overline{T^{a F, a G}}} e^{w_N(\lvert \mathrm{Re \ }\xi\rvert)} \lvert f^{(\alpha)}(\xi)\rvert < \infty, \qquad \forall N \in \mathbb{N}, \alpha \in \mathbb{N}^2, a \in (0,1).
$$
We endow $\mathcal{K}_{W}(T^{F,G})$ with the locally convex topology generated by the system of seminorms $\{ p_{N,\alpha,a} \mid N \in \mathbb{N}, \alpha \in \mathbb{N}^2, a \in (0,1)\}$.  We set
$$
\mathcal{U}_{W}(T^{F,G}) = \mathcal{K}_{W}(T^{F,G}) \cap \mathcal{H}(T^{F,G})
$$
and endow it 
 with the subspace topology induced by $\mathcal{K}_{W}(T^{F,G})$.\end{definition}

\smallskip

Note that $\mathcal{K}_{W}(T^{F,G})$ is a Fr\'echet space. Moreover, since $\mathcal{U}_{W}(T^{F,G})$ is closed in $\mathcal{K}_{W}(T^{F,G})$, we obtain that $\mathcal{U}_{W}(T^{F,G})$ is a Fr\'echet space as well.

We now give a useful derivative-free characterization of the Fr\'{e}chet space \(\mathcal{U}_W(T^{F,G})\). 

\begin{lemma}\label{lemma-df}
 Let \(W = (w_N)_{N \in \mathbb{N}}\) be a weight function system that satisfies \((\alpha)\) and let \(F, G \in \mathcal{F}(\mathbb{R})\). A function $f \in \mathcal{H}(T^{F,G})$ belongs to 
 $\mathcal{U}_W(T^{F,G})$ if and only if 
$$
p_{N,a}(f) =  \sup_{\xi \in \overline{T^{a F, a G}}} e^{w_N(\lvert \mathrm{Re \ }\xi\rvert)} \lvert f(\xi)\rvert < \infty, \qquad \forall N \in \mathbb{N}, a \in (0,1).
$$
Moreover, the topology of $\mathcal{U}_{W}(T^{F,G})$ is  generated by the system of seminorms $\{ p_{N,a} \mid N \in \mathbb{N}, a \in (0,1)\}$.
\end{lemma}
\begin{proof}
Let $\widetilde{\mathcal{U}}_W(T^{F,G})$ be the space consisting of all $f \in  \mathcal{H}(T^{F,G})$ such that $p_{N,a}(f) < \infty$ for all $N \in \mathbb{N}$ and $a \in (0,1)$, endowed with the topology generated by the system of seminorms $\{ p_{N,a} \mid N \in \mathbb{N}, a \in (0,1)\}$.
We need to show that  $\mathcal{U}_W(T^{F,G}) = \widetilde{\mathcal{U}}_W(T^{F,G})$ as locally convex spaces. It is clear that  $\mathcal{U}_W(T^{F,G}) \subseteq \widetilde{\mathcal{U}}_W(T^{F,G})$ continuously. We now show that $\widetilde{\mathcal{U}}_W(T^{F,G})$ is continuously included in $\mathcal{U}_W(T^{F,G})$. Take \(a \in (0,1)\), $\alpha \in \mathbb{N}^2$, and \(N\in \mathbb{N}\) arbitrary. Fix $ b \in (a,1)$ and choose $M > N$ and \(A > 1 \) as in Lemma \ref{Subadditivity}. 
By Lemma \ref{Lemma 1}, there is \(\varepsilon > 0\) such that  $\overline{T^{a F, a G}} + \overline{B}(0, \varepsilon) \subseteq T^{bF,bG}$. Let $f \in \widetilde{\mathcal{U}}_W(T^{F,G})$ be arbitrary. The Cauchy integral formula yields, %that
 for all $\xi \in \overline{T^{a F, a G}}$ (we  write $k = |\alpha|$),
$$
\lvert f^{(\alpha)}(\xi) \rvert =  \bigg\lvert \frac{k!}{2\pi i} \oint_{|z-\xi|=\varepsilon}\frac{f(z)}{(z-\xi)^{k+1}}\mathrm{d}z \bigg \rvert \leq  \frac{k!}{2\pi \varepsilon^{k+1}} \oint_{|z-\xi|=\varepsilon}\lvert f(z) \rvert \lvert \mathrm{d} z \rvert.
$$
Hence,
 \begin{align*}
p_{N,\alpha,a}(f) &=  \sup_{\xi \in \overline{T^{a F, a G}}} e^{w_N(\lvert \mathrm{Re \ }\xi\rvert)} \lvert f^{(\alpha)}(\xi)\rvert  
\leq     \frac{k!}{2\pi \varepsilon^{k+1}} \sup_{\xi \in \overline{T^{a F, a G}}}e^{w_N(\lvert \mathrm{Re \ }\xi\rvert)}\oint_{|z-\xi|=\varepsilon}\lvert f(z) \rvert \lvert \mathrm{d} z \rvert \\
&\leq     \frac{Ak!e^{w_M(\varepsilon)}}{2\pi \varepsilon^{k+1}} \sup_{\xi \in \overline{T^{a F, a G}}}\oint_{|z-\xi|=\varepsilon} e^{w_M(\lvert \mathrm{Re \ }z\rvert)}\lvert f(z) \rvert \lvert \mathrm{d} z \rvert \leq     \frac{Ak!e^{w_M(\varepsilon)}}{\varepsilon^{k}} p_{M,b}(f). 
\end{align*}
\end{proof}\subsection{The abstract Mittag-Leffler lemma}\label{ML-subs}
In this subsection, we explain the abstract  Mittag-Leffler lemma for projective spectra of Fr\'echet spaces. This will be needed for
 the proof of our main result. We follow the  book \cite{wengenroth2003derived}.

\begin{definition}
    A \emph{projective spectrum} \(\mathcal{X} = (X_n,\rho_m^n)\) is a sequence  \((X_n)_{n\in \mathbb{N}}\) of vector spaces together with linear maps \(\rho_m^n \colon X_m \rightarrow X_n\), \(n\leq m\), such that \(\rho_n^n = \operatorname{id}_{X_n}\) and \(\rho_k^n = \rho_m^n \circ \rho_k^m\) for \(n\leq m \leq k\). We call \(\rho_m^n\) the \emph{linking maps}. Consider the map
  \[
  \Psi \colon \prod_{n\in\mathbb{N}} X_n \rightarrow \prod_{n\in\mathbb{N}} X_n \colon (x_n)_{n\in \mathbb{N}} \mapsto (x_n - \rho_{n+1}^n(x_{n+1}))_{n\in\mathbb{N}}.
  \]
We define the \emph{projective limit} of \(\mathcal{X}\) as the kernel of $\Psi$, that is,
    \[
    \operatorname{Proj } \mathcal{X} = \left\{(x_n)_{n\in \mathbb{N}}\in \prod_{n\in \mathbb{N}} X_n \mid \rho_m^n(x_m) = x_n, \forall m\geq n\right\} .
    \]
Furthermore, we define the (first)  \emph{derived projective limit} of  \(\mathcal{X}\) as the cokernel of \(\Psi\), namely,
    \[
    \operatorname{Proj}^1 \mathcal{X} = \prod_{n\in \mathbb{N}}X_n / \operatorname{im } \Psi .
    \]
     \end{definition}
  Next, we introduce morphisms between projective spectra.
 \begin{definition}
    Let \(\mathcal{X} = (X_n, \rho_m^n)\) and \(\mathcal{Y} = (Y_n, \sigma_m^n)\) be two projective spectra. A \emph{morphism} \(f = (f_n)_{n\in\mathbb{N}} \colon \mathcal{X} \rightarrow \mathcal{Y}\) consists of linear maps \(f_n \colon X_n \rightarrow Y_n\) such that \(f_n \circ \rho_m^n = \sigma_m^n \circ f_m\) for all \(m \geq n\). The \emph{kernel} \(\operatorname{ker }f\) of the morphism \(f\) is defined as the projective spectrum \((\operatorname{ker }f_n)_{n\in \mathbb{N}}\) with linking maps \({\rho_m^n}_{\vert \operatorname{ker }X_m}\colon \operatorname{ker } X_m \rightarrow \operatorname{ker }X_n\). Moreover, we define for each morphism \(f\colon \mathcal{X}\rightarrow \mathcal{Y}\) the linear map \(\operatorname{Proj }f \colon \operatorname{Proj } \mathcal{X} \rightarrow \operatorname{Proj }\mathcal{Y}\colon (x_n)_{n\in\mathbb{N}} \mapsto (f_n(x_n))_{n\in \mathbb{N}}\).
 \end{definition}
We shall use the following basic property of the derived projective limit.
 \begin{proposition}[{\cite[Proposition 3.1.8]{wengenroth2003derived}}] 
 \label{Mittag-Leffler}
     Let \(\mathcal{X} = (X_n, \rho_m^n)\) and \(\mathcal{Y} = (Y_n, \sigma_m^n)\) be two projective spectra and let \(f = (f_n)_{n\in \mathbb{N}} \colon \mathcal{X}\rightarrow \mathcal{Y}\) be a morphism. The linear map \(\operatorname{Proj } f \colon \rightarrow \operatorname{Proj }\mathcal{X} \rightarrow \operatorname{Proj }\mathcal{Y}\) is surjective if 
     \begin{enumerate}[label=(\roman*)]
         \item For all \(n \in \mathbb{N}\) there is \(m > n\) such that \(\sigma_m^n(Y_m) \subseteq \operatorname{im }f_n\).
         \item \(\operatorname{Proj }^1 \operatorname{ker }f = \{0\}\).
     \end{enumerate}
 \end{proposition}
 In practice, a projective spectrum \(\mathcal{X} = (X_n, \rho_m^n)\)  often consists of locally convex spaces $X_n$ and continuous linking maps $\rho_m^n$. In this case, there exist sufficient (and necessary) linear topological conditions on $\mathcal{X}$ to decide whether \(\operatorname{Proj }^1 \mathcal{X} = \{0\}\); see \cite[Section 3.2]{wengenroth2003derived} for a detailed overview of such conditions. We will need the following sufficient condition for  \(\operatorname{Proj }^1 \mathcal{X} = \{0\}\) for projective spectra \(\mathcal{X}\) consisting of Fr\'echet spaces. This result is sometimes called the \emph{abstract Mittag-Leffler lemma}.
  \begin{proposition}\label{density condition}
Let \(\mathcal{X} = (X_n, \rho_m^n)\) be a projective spectrum consisting of Fréchet spaces and continuous  linking maps. If
\begin{equation}
\label{suff-cond}
\forall n \in \mathbb{N} \, \exists m > n \, \forall k > m: \rho_m^n(X_m) \subseteq \overline{\rho_k^n(X_k)}^{X_n},
\end{equation}
then \(\operatorname{Proj }^1 \mathcal{X} = \{0\}\).
 \end{proposition}
 \begin{proof}
     This follows from \cite[Theorem 3.2.1]{wengenroth2003derived}.
 \end{proof}
 
  \section{Surjectivity of the Cauchy-Riemann and Laplace operators on \(\mathcal{K}_{W}(T^{F,G})\)}\label{main section}
  
In this section, we state our main result, which characterizes  the surjectivity of both the Cauchy-Riemann and  Laplace operators on the spaces \(\mathcal{K}_W(T^{F,G})\) in terms of the defining weight function system $W$. Furthermore, we also discuss two particular classes of weight function systems that are generated by a single weight function and for which our main result is directly applicable.

 \begin{theorem}\label{main theorem}Let \(F, G \in \mathcal{F}(\mathbb{R})\) and 
    let \(W\) be a weight function system that satisfies $(\alpha)$ and \((N)\). 
    Then, the following  statements are  equivalent:
    \begin{enumerate}[label=(\roman*)]
        \item The Cauchy-Riemann operator \(\overline{\partial} \colon \mathcal{K}_{W}(T^{F,G}) \rightarrow \mathcal{K}_{W}(T^{F,G})\) is surjective.
        \item The Laplace operator \(\Delta\colon \mathcal{K}_{W}(T^{F,G}) \rightarrow \mathcal{K}_{W}(T^{F,G})\) is surjective.
        \item The weight function system \(W\) satisfies \((\epsilon)_0\).
        \item The space \(\mathcal{U}_{W}(T^{F,G})\) is non-trivial.
    \end{enumerate}
\end{theorem}
The proof of Theorem \ref{main theorem} will be
given in Section \ref{proof main theorem}. 
The most involved part will be the proof of the implication   \((iii)\implies (i)\), for which we will make use of  the abstract Mittag-Leffler lemma (Proposition \ref{density condition}). In fact, to check the statement \eqref{suff-cond}, we will 
employ a Runge type approximation theorem related to the spaces  \(\mathcal{U}_{W}(T^{F,G})\), shown in  Section \ref{Runge type result}.

Let us now introduce two classes of weight function systems solely generated by a single weight function \(w\) and rephrase the conditions appearing in Theorem \ref{main theorem} in terms of  \(w\). To this end, we consider the following conditions on  $w$:
\begin{enumerate}
    \item [\((\alpha)\)] \(\exists C > 0, A > 1 \, \forall t \geq 0: w(2t) \leq Cw(t) + \log A\) .\\
    \item [\((\epsilon)_0\)] \(\forall \mu > 0: \displaystyle \int_0^{\infty} w(t) e^{-\mu t}\mathrm{d}t\) .
    \item [\((N)_i\)] \(\exists C > 1: \displaystyle \int_0^{\infty} e^{w(t) - w(Ct)}\mathrm{d}t < \infty .\)
    \item [\((N)_o\)] \(\exists C > 0: \displaystyle\int_0^{\infty} e^{-Cw(t)}\mathrm{d}t < \infty\).\\
\end{enumerate}

Given a weight function $w$, we define the weight function system \(W_w = (w(N \cdot))_{N\in\mathbb{N}}\). 
The following lemma provides, among other useful properties, a sufficient condition on \(w\) for \(W_w\) to satisfy \((N)\).

\begin{lemma}\label{inside}
    Let \(w\) be a weight function.    Then, 
    \begin{enumerate}[label=(\roman*)]
        \item \(W_w\) always satisfies \((\alpha)\).
         \item \(W_w\) satisfies \((\epsilon)_0\) if and only if \(w\) satisfies \((\epsilon)_0\).
	\item \(W_w\) satisfies \((N)\) if and only if \(w\) satisfies \((N)_i\).
        \item If \(w\) satisfies
       \begin{equation}
       \label{deltacond}
        \exists C>0,A > 1\, \forall t \geq 0: 2w(t) \leq w(Ct) + \log A,
        \end{equation} then \(w\) satisfies \((N)_i\).
    \end{enumerate}
 \end{lemma}
\begin{proof}
Statements \((i)\)--\((iii)\) are obvious,  so we only show \((iv)\).
 Condition \eqref{deltacond} implies that there is \(A > 1\) such that \(2 w(t/A) \leq w(t)+\log A\) for all \(t \geq 0\). 
Iterating this condition yields  \(2^n w(t/A^n)\leq w(t) +(2^n-1) \log A \) for all \(t\geq 0\) and  $n \in \mathbb{N}$. Choose \(t_0 > 0\) such that \(w(t_0) > \log A\). Take \(t\geq t_0\) arbitrary and pick \(n\in \mathbb{N}\) such that \(t_0 A^n \leq t < t_0 A^{n+1}\). As $w$ is non-decreasing, we find that
        \begin{align*}
        w(t) &\geq w(t_0 A^n) \geq 2^n  w(t_0) - (2^n-1) \log A \\ \nonumber
                  &= 2^n (w(t_0)-\log A)+ \log A 
                  \geq 2^{\frac{\log \frac{t}{t_0}}{\log A}-1} (w(t_0)-\log A)+\log A \nonumber .
    \end{align*}
In particular, there is \(\tau > 0\) such that \(w(t) \geq 2\log (1+t)\) for \(t\geq \tau\). Hence, we get  \(w(t) + 2\log(1+t) \leq 2w(t) \leq w(At) + \log A\) for all \(t \geq \tau\).
This implies that \(w\) satisfies \((N)_i\) with \(C = A\).
\end{proof}
Hence, we can reformulate Theorem \ref{main theorem} for weight function systems of type \(W_w\) as follows.
\begin{corollary}\label{main theorem for N inside} Let \(F, G \in \mathcal{F}(\mathbb{R})\) and let the weight function $w$ satisfy \((N)_i\).
   Then, the following  statements are  equivalent:
     \begin{enumerate}[label=(\roman*)]
        \item The Cauchy-Riemann operator \(\overline{\partial} \colon \mathcal{K}_{W_w}(T^{F,G}) \rightarrow \mathcal{K}_{W_w}(T^{F,G})\) is surjective.
        \item The Laplace operator \(\Delta\colon \mathcal{K}_{W_w}(T^{F,G}) \rightarrow \mathcal{K}_{W_w}(T^{F,G})\) is surjective.
        \item The weight function \(w\) satisfies \((\epsilon)_0\).
        \item The space \(\mathcal{U}_{W_w}(T^{F,G})\) is non-trivial.
    \end{enumerate}
\end{corollary}
Note that,  when $h<\infty$, Theorem \ref{theorem 1} from the Introduction is a particular instance of Corollary \ref{main theorem for N inside}.
\begin{example} 
Suppose that either \(a > 0\) and \(b \in \mathbb{R}\) or \(a=0\) and \(b\geq0\). 
Consider a weight function \(w\) defined as $w(t)= e^{t^a (\log(e+t))^b}$ for \(t\) sufficiently large. Then, by Lemma \ref{inside}$(iv)$, $w$ always satisfies $(N)_i$, while it satisfies \((\epsilon)_0\) if and only if \(0\leq a < 1\) or \(a=1\) and \(b<0\). Let \(F, G \in \mathcal{F}(\mathbb{R})\). By Corollary \ref{main theorem for N inside}, we obtain that the \(\overline{\partial}\)-operator and/or the \(\Delta\)-operator are surjective on  $\mathcal{K}_{W_w}(T^{F,G})$ if and only if \(0\leq a < 1\) or \(a=1\) and \(b<0\).
\end{example}

Next,  we consider weight functions systems of the form \(\tilde{W}_w = (Nw)_{N\in\mathbb{N}}\), where $w$ is a weight function.
\begin{lemma}\label{lemma 8}
    Let \(w\) be a weight function. Then, 
    \begin{enumerate}[label=(\roman*)]
        \item \(\tilde{W}_w\) satisfies \((\alpha)\) if and only if \(w\) satisfies \((\alpha)\).
        \item \(\tilde{W}_w\) satisfies \((\epsilon)_0\) if and only if \(w\) satisfies \((\epsilon)_0\).
        \item  \(\tilde{W}_w\) satisfies \((N)\) if and only if \(w\) satisfies \((N)_o\). 
                \item If \(w\) satisfies \((\alpha)\), then there is $a >0$ such that $\omega(t) = O(t^a)$ as $t \to \infty$.  In particular, \(w\) satisfies \((\epsilon)_0\).
    \end{enumerate}
\end{lemma}
\begin{proof}
    Statements \((i)\)--\((iii)\) are obvious. 
   Property $(iv)$ is  a consequence of general results from the theory of regular variation \cite{BGT}, but since the proof is short, we give it here for the sake of completeness. Since \(w\) satisfies \((\alpha)\), there are $C >1$ and \(A > 1\) such that \(w(2t) \leq Cw(t) + \log A\) for all \(t\geq 0\).         
        Take an arbitrary \(t \geq 1\)  and choose \(n \in \mathbb{N}\) such that \(2^n \leq t < 2^{n+1}\). By iteration of the above inequality, we get that
    \[
    w(t) \leq w(2^{n+1}) \leq C^{n+1}w(1) + \frac{C^{n+1}-1}{C-1} \log A \leq  C^{\frac{\log t}{\log 2}+1}w(1) + \frac{C^{\frac{\log t}{\log 2}+1}-1}{C-1}\log A,
    \]
which shows the result.
\end{proof}
Hence, we can reformulate Theorem \ref{main theorem} for weight function systems of type \(\tilde{W}_w\) as follows.
\begin{corollary}\label{main theorem for N outside} Let \(F, G \in \mathcal{F}(\mathbb{R})\) and
let \(w\) be a weight function that satisfies \((\alpha)\) and \((N)_o\). 
Then, the following  statements are  valid:
\begin{enumerate}[label=(\roman*)]
        \item The Cauchy-Riemann operator \(\overline{\partial} \colon \mathcal{K}_{\tilde{W}_w}(T^{F,G}) \rightarrow \mathcal{K}_{\tilde{W}_w}(T^{F,G})\) is surjective.
        \item The Laplace operator \(\Delta\colon \mathcal{K}_{\tilde{W}_w}(T^{F,G}) \rightarrow \mathcal{K}_{\tilde{W}_w}(T^{F,G})\) is surjective.
        \item The space \(\mathcal{U}_{\tilde{W}_w}(T^{F,G})\) is non-trivial.
        \end{enumerate}
\end{corollary}

\begin{example} 
Let \(a > 0\) and \(b \in \mathbb{R}\) or $a = 0$ and $b \geq 1$. Consider a weight function \(w\) defined as $w(t)= t^a (\log(e+t))^b$ for \(t\) sufficiently large. Then, $w$ always satisfies \((\alpha)\) and $(N)_o$. Let \(F, G \in \mathcal{F}(\mathbb{R})\). By Corollary \ref{main theorem for N outside}, we obtain that the \(\overline{\partial}\)-operator and the \(\Delta\)-operator are surjective on  $\mathcal{K}_{\tilde{W}_w}(T^{F,G})$. 
\end{example}

We end this section with an important remark. 

\begin{remark}
\label{rm infinity case}
Let \(W = (w_N)_{N\in\mathbb{N}}\) be a weight function system.
     Define \(\mathcal{K}_{W}(\mathbb{C}) = \operatorname{Proj }_{N\in \mathbb{N}} \mathcal{K}_{w_N}(T_N)\) with the natural restriction maps as linking maps. Assume that \(W\) satisfies \((\alpha)\) and $(N)$. Similarly to Theorem \ref{main theorem}, we have that the following statements are equivalent:
     \begin{enumerate}[label=(\roman*)]
          \item The Cauchy-Riemann operator \(\overline{\partial} \colon \mathcal{K}_{W}(\mathbb{C}) \rightarrow \mathcal{K}_{W}(\mathbb{C})\) is surjective.
        \item The Laplace operator \(\Delta\colon \mathcal{K}_{W}(\mathbb{C}) \rightarrow \mathcal{K}_{W}(\mathbb{C})\) is surjective.
        \item The weight function system \(W\) satisfies \((\epsilon)_0\).
        \item The space \(\mathcal{U}_{W}(\mathbb{C})\) is non-trivial. 
     \end{enumerate}
This can be shown in the same way as Theorem \ref{main theorem} (one also needs an obvious modification of Proposition \ref{non-triviality kernel} below); the details are left to the reader. Furthermore,  Corollaries \ref{main theorem for N inside} and \ref{main theorem for N outside} are still valid if we replace \(T^{F,G}\) by \(\mathbb{C}\). Finally, note that, for $h = \infty$, Theorem \ref{theorem 1} from the Introduction is a particular instance of this modified version of Corollary \ref{main theorem for N inside}.

\end{remark}

\section{A weighted version of the Runge approximation result}\label{Runge type result}
This section is devoted to show the following weighted Runge approximation result.
\begin{theorem}\label{approximation theorem}
 Let \(F, G \in \mathcal{F}(\mathbb{R})\) and let \(W = (w_N)_{N \in \mathbb{N}}\) be a weight function system that satisfies \((\alpha)\), $(N)$, and $(\epsilon)_{0}$. Then, for all \(N \in \mathbb{N}\) there is \(M > N\) such that for all \(K > M\) and for all \(1<a<b<\infty\):
 \begin{equation}\label{equation 4.1}
\forall \varepsilon >0 \: \forall f \in \mathcal{U}_{w_M}(T^{a F, a G}) \, \exists g \in \mathcal{U}_{w_K}(T^{b F, b G}): \sup_{\xi \in \overline{T^{F, G}}} e^{w_N(\lvert \operatorname{Re }\xi\rvert)} \lvert f(\xi) - g(\xi)\rvert \leq \varepsilon .
 \end{equation}
 \end{theorem}
 Our proof of Theorem \ref{approximation theorem} is a variant of the standard proof of the classical Runge theorem. We need the following pole pushing lemma.

  \begin{lemma}
  \label{Pole-pushing lemma}
    Let \(K \subseteq \mathbb{C}\) be closed and let \(\alpha, \beta \in \mathbb{C}\setminus K\). Suppose that \(\alpha\) and \(\beta\) lie in the same connected component of \(\mathbb{C}\setminus K\). Then, for all \(\varepsilon > 0\) there is a rational function \(R\) with \(\beta\) as only pole such that
    \[
    \sup_{\xi \in K}\left\lvert \frac{1}{\xi - \alpha} - R(\xi)\right\rvert \leq \varepsilon .
    \]
 \end{lemma}
 \begin{proof}
 This result is well-known, see 
 for instance the second part of the proof of \cite[Lemma 5.10]{Stein2003ComplexAnalysis}.
 \end{proof}
  We are now able to prove Theorem \ref{approximation theorem}.
  \begin{proof}[Proof of Theorem \ref{approximation theorem}]
Let \(N \in \mathbb{N}\)  be arbitrary. By Lemma \ref{Subadditivity}, there are \(L > N\) and \(A > 1\) such that, for all \(t,s \geq 0\),
  \[
  w_N(t+s) \leq w_L(t) + w_L(s) + \log A .
  \]
  Condition \((N)\) tells us that there is \(M > L\) such that 
  \begin{equation}
  \label{eq: N in use}
  \int_0^{\infty} e^{w_L(t) - w_M(t)}\mathrm{d}t < \infty .
  \end{equation}
   Take   arbitrary  \(K > M\) and \(1<a<b<\infty\). 
    Fix \(h > 2b\max\{\sup_{t \in \mathbb{R}} F(t), \sup_{t \in \mathbb{R}} G(t)\}\) and pick an arbitrary \(f \in \mathcal{U}_{w_M}(T^{a F, a G})\).  By Lemma \ref{Subadditivity}, there are \(\tilde{K} > K\) and \(C > 1\) such that, for all  
\(t, s \geq 0\),
\[
  w_K(t+s) \leq w_{\tilde{K}}(t) + w_{\tilde{K}}(s) + \log C. 
\]
 Lemma \ref{Lemma 2} 
allows us to find
  \(Q \in \mathcal{H}(T_h)\) such that $\lvert Q(\xi) \rvert   \geq e^{w_{\tilde{K}}(\lvert \mathrm{Re \ }\xi\rvert)} $
 for all \(\xi \in T_h\). Define 
 \(P = Q(0)/Q\). Fix $c \in (1,a)$. By Lemma \ref{smooth interpolation} (applied to $cF$ and $cG$ and with \(1/c\) playing the role of \(a\) in Lemma \ref{smooth interpolation}), we can choose \(r> 0\)  and \(\Phi, \Psi \in \mathcal{F}( \mathbb{R}) \cap C^{\infty}(\mathbb{R})\), with all derivatives bounded, such that \(\overline{T^{F,G}} + \overline{B}(0,r) \subseteq T^{\Phi, \Psi}\) and \(\overline{T^{\Phi,\Psi}} \subseteq T^{cF,cG}\). We now show that we can represent \(f(\xi)\) as an (improper) contour integral along the boundary of \(T^{\Phi,\Psi}\) for all \(\xi \in \overline{T^{F,G}}\). More precisely, we have: 
  \begin{claim 1}\emph{
  Define the contours \(\Gamma_1 = \{t+i\Phi(t)\mid t\in\mathbb{R}\}\) and \(\Gamma_2 = \{t-i\Psi(t)\mid t \in \mathbb{R}\}\). Orient \(\Gamma_1\) from \enquote*{right to left} and \(\Gamma_2\) from \enquote*{left to right}. Then, \(f(\xi) = I_{1}(\xi)+I_{2}(\xi)\) for all \(\xi \in \overline{T^{F,G}}\)  
  , where
    $$
     I_k(\xi) = \frac{1}{2\pi i} \int_{\Gamma_k} \frac{f(z)P(z-\xi)}{z-\xi}\mathrm{d}z, \qquad k = 1,2.
      $$
      }
      \end{claim 1}
       \begin{claimproof}
     Let  \(\xi \in \overline{T^{F, G}}\) be arbitrary.   For \(R > \lvert \mathrm{Re \ }\xi\rvert\),  we set       \begin{align*}
      \Gamma_1^R &= \{t+i\Phi(t)\vert t \in [-R,R]\}, \\
      \Gamma_2^R &= \{t-i\Psi(t)\vert t \in [-R,R]\}, \\
      \Gamma_3^R &= [R-i\Psi(R), R+i\Phi(R)], \\
      \Gamma_4^{R} &= [-R-i\Psi(-R), -R+i\Phi(-R)] ,
      \end{align*}
      where we orient the closed contour $\bigcup_{i=1}^{4} \Gamma^{R}_{i}$ counterclockwise.       Since \(P(0) = 1\), the Cauchy integral formula yields 
\begin{equation*}
      f(\xi) 
      = \sum_{k=1}^4 I_k(\xi,R) ,
      \end{equation*}
      where 
  $$
    I_k(\xi,R) = \frac{1}{2\pi i}\int_{\Gamma_k^R} \frac{f(z)P(z-\xi)}{z-\xi}\mathrm{d}z, \qquad k  =1,2,3,4 .
$$
It suffices to show that  $\lim_{R\to\infty} I_k(\xi,R) =I_k(\xi)$ for $k = 1,2$ and  $\lim_{R\to\infty}  I_k(\xi,R) =0$ for $k =3,4$. We only consider $k =1,3$, as the other cases can be treated similarly. We have, as $R\to\infty$, 
\begin{align*}
&\lvert I_1(\xi)-I_1(\xi,R)\rvert \leq e^{w_N(\lvert \operatorname{Re} \xi\rvert)}\lvert I_1(\xi)-I_1(\xi,R)\rvert \\ &\leq \frac{A}{2\pi}\int_{t\in \mathbb{R}, \lvert t\rvert \geq R} \frac{e^{w_L(\lvert t \rvert)}\lvert f(t+i\Phi(t))\rvert e^{w_L(\lvert t - \operatorname{Re}\xi\rvert)}\lvert P(t+i\Phi(t)-\xi)\rvert}{\lvert t+i\Phi(t) - \xi\rvert} \lvert 1+i\Phi^{\prime}(t)\rvert \mathrm{d}t\\
&\leq \frac{A \lvert Q(0)\rvert}{2\pi r} (1+\sup_{s\in\mathbb{R}}\lvert \Phi^{\prime}(s)\rvert) \sup_{z\in \overline{T^{cF, cG}}}(e^{w_M(\lvert \operatorname{Re}z\rvert)} \lvert f(z)\rvert)  \int_{t\in\mathbb{R}, \lvert t \rvert \geq R} e^{w_L(\lvert t \rvert) - w_M(\lvert t \rvert)} \to 0 ,
\end{align*}
in view of \eqref{eq: N in use},
 and 
\begin{align*}
\lvert I_3(\xi,R)| &\leq \frac{\lvert Q(0)\rvert}{2\pi}\int_{\Gamma_3}\frac{\lvert f(z)\rvert}{\lvert z-\xi\rvert}\lvert \mathrm{d}z\rvert \\&\leq \frac{
\lvert Q(0)\rvert}{2\pi}\sup_{z \in \overline{T^{cF,cG}}} \lvert f(z)\rvert\frac{\sup_{t \in \mathbb{R}} \Phi(t) + \sup_{t \in \mathbb{R}} \Psi(t)}{R-\lvert \mathrm{Re \ }\xi\rvert} \to 0.  
\end{align*}
\end{claimproof}

An inspection of the proof of the previous Claim shows for \(k=1,2\), as \(R \to \infty\), \[\sup_{\xi\in \overline{T^{F,G}}}e^{w_N(\lvert \operatorname{Re}\xi\rvert)}\lvert I_k(\xi) - I_k(\xi,R)\rvert \to 0 .\] Hence,      \begin{equation}\label{uniform integral}
    \lim_{R\to \infty}\sup_{\xi \in \overline{T^{F, G}}}e^{w_N(\lvert \operatorname{Re} \xi\rvert)}\bigg\lvert  f(\xi) - 
I_1(\xi,R)-I_2(\xi,R)\bigg\rvert 
    = 0 .
    \end{equation}
    The idea is now to use the integral representation we just proved to approximate \(f\)  uniformly on \(\overline{T^{ F, G}}\), with respect to the weight \(e^{w_N}\), via Riemann sums. Hereafter, we will approximate these Riemann sums uniformly, again with respect to \(e^{w_N}\), on \(\overline{T^{F, G}}\) with functions belonging to \(\mathcal{U}_{w_K}(T^{b F, b G})\) by using Lemma \ref{Pole-pushing lemma} (pole pushing). From now on we fix an arbitrary \(\varepsilon > 0\).
        \begin{claim 2}
        \emph{
        There are \(J \in \mathbb{N}\), 
        \(z_1,\ldots,z_J  \in \Gamma_1 \cup \Gamma_2\), 
         and  \(C_1,\ldots,C_J \in \mathbb{C}\)  such that
        \[
        \sup_{\xi \in \overline{T^{F, G}}} e^{w_N(\lvert \operatorname{Re} \xi \rvert)}\bigg\lvert f(\xi) - \sum_{j=1}^J \frac{C_j  P(z_j-\xi)}{z_j-\xi}\bigg\rvert \leq \frac{\varepsilon}{2} .
        \]}
    \end{claim 2}
    \begin{claimproof}
        Equation \eqref{uniform integral} implies that there is \(R>0\) such that
        \[
        \sup_{\xi \in \overline{T^{F, G}}} e^{w_N(\lvert \operatorname{Re} \xi\rvert)}\bigg\lvert f(\xi) -  I_1(\xi,R)-I_2(\xi,R)\bigg\rvert \leq \frac{\varepsilon}{4} .
        \]
        Hence, it suffices to show, for $k = 1,2$, that there are \(J_k \in \mathbb{N}\),  \(z_1,\ldots,z_{J_k} \in \Gamma_{k}\) and  \(C_1,\ldots,C_{J_k} \in \mathbb{C}\)  such that
 \begin{equation}
  \label{claim2}
        \sup_{\xi \in \overline{T^{F, G}}} e^{w_N(\lvert \operatorname{Re}\xi\rvert)}\bigg\lvert I_k(\xi,R)- \sum_{j=1}^{J_k} \frac{C_j  P(z_j-\xi)}{z_j-\xi}\bigg\rvert \leq \frac{\varepsilon}{8} .
 \end{equation}
 We only consider $k=1$, as $k =2$ is completely analogous.
        For convenience, define 
        \[H_1(t,\xi) = -\frac{1}{2\pi i}\frac{f(t+i\Phi(t))P(t+i\Phi(t)-\xi)}{t+i\Phi(t)-\xi} (1+i\Phi^{\prime}(t))\] 
                for \((t,\xi) \in [-R,R]\times \overline{T^{F, G}}\).
 Then, for all \(t\in [-R,R]\) and \(\xi \in \overline{T^{F, G}}\) with \(\lvert \operatorname{Re }\xi \rvert > R\),
\begin{align*}
         e^{w_N(\lvert \operatorname{Re}\xi \rvert)}\lvert H_1(t,\xi) \rvert  &\leq \frac{A \lvert Q(0)\rvert \sup_{z \in \overline{T^{cF, cG}}} (e^{w_M(\lvert \mathrm{Re \ }z\rvert)} \lvert f(z)\rvert)}{2\pi (\lvert \mathrm{Re \ }\xi\rvert-R)} (1+\sup_{t\in\mathbb{R}} \lvert \Phi^{\prime}(t)\rvert) \\
            &\to 0 \text{ as } \lvert \operatorname{Re} \xi\rvert \to \infty.
\end{align*}
Hence, there is \(B > R\) such that, for all \(t \in [-R,R]\) and \(\xi \in \overline{T^{F,G}}\) with \(\lvert \operatorname{Re} \xi \rvert > B\),
\[e^{w_N(\lvert \operatorname{Re}\xi\rvert)} \lvert H_1(t,\xi)\rvert \leq \frac{\varepsilon}{32R}.\]
The continuity of \(H_1\) insures that there is \(\delta > 0\) such that,
for all \(t,t^{\prime}\in [-R,R]\) and \(\xi \in \overline{T^{F,G}}\) with \(\lvert \operatorname{Re}\xi \rvert \leq B\), 
\[
\lvert t - t^{\prime}\rvert \leq \delta \implies e^{w_N(\lvert \operatorname{Re} \xi\rvert)}\lvert H_1(t,\xi)-H_1(t^{\prime},\xi)\rvert 
\leq \frac{\varepsilon}{16R} .
\]

            Choose now \(J_1 \in \mathbb{N}\) so large that \(2R / J_1\leq \delta\). Define \(t_j = -R +  2jR/J_1\) for \(j= 1,\ldots,J_1\).         Then,
      \begin{align*}
            &\sup_{\xi \in \overline{T^{F, G}}}  e^{w_N(\lvert \operatorname{Re} \xi\rvert)}\bigg\lvert I_1(\xi,R)-\sum_{j=1}^{J_1} H_1(t_j,\xi)\frac{2R}{J_1} \bigg\rvert \\
            &\leq   \sup_{\xi \in \overline{T^{F, G}}}e^{w_N(\lvert \operatorname{Re}\xi\rvert)}
            \bigg\lvert \int_{-R}^R H_1(t,\xi)\mathrm{d}t - \sum_{j=1}^{J_1} H_1(t_j,\xi) \frac{2R}{J_1}\bigg\rvert 
            \leq \sum_{j=1}^{J_1} \frac{\varepsilon}{16R}  \frac{2R}{J_1} 
            = \frac{\varepsilon}{8}.
            \end{align*}
This shows \eqref{claim2} for $k =1$.
     \end{claimproof}
     
    Next, we  approximate the function  \(\displaystyle \sum_{j=1}^J C_j P(z_j - \xi)(z_j-\xi)^{-1}\) from the previous Claim uniformly on \(\overline{T^{F, G}}\) with respect to the weight \(e^{w_N}\).
    \begin{claim 3}\emph{
    There are rational functions $R_1, \ldots, R_J$ such that
    \begin{enumerate}[label=(\roman*)]
    \item For all $j = 1, \ldots, J$ we have  
    $$R_j(z) = \sum_{k=1}^{d_j} \gamma_{j,k}(z-\alpha_j)^{-k}
    $$
    with  \(\alpha_j \in \mathbb{C}\setminus \overline{T^{b F, b G}}\),  \(\gamma_{j,k} \in \mathbb{C}\), and \(d_j \in \mathbb{N}\).  In particular, \(\sup_{z\in \overline{T^{b F, b G}}} \lvert R_j(z)\rvert < \infty\) for all  $j = 1, \ldots, J$.
    \item  \(\ \displaystyle \operatorname*{sup}_{\xi \in \overline{T^{F, G}}}e^{w_N(\lvert \operatorname{Re}\xi\rvert)}\big\lvert \sum_{j=1}^J C_j  P(z_j-\xi)(z_j-\xi)^{-1}-\sum_{j=1}^J C_j  P(z_j-\xi)  R_j(\xi) \big\rvert \leq \displaystyle \frac{\varepsilon}{2}\).
    \end{enumerate}}
\end{claim 3}
\begin{claimproof}
    It holds that  \[D =  \max_{1\leq j \leq J}\sup_{\xi \in \overline{T^{F, G}}} e^{w_N(\lvert \operatorname{Re}\xi\rvert)}\lvert C_jP(z_j-\xi)\rvert \leq  A \lvert Q(0)\rvert \max_{1\leq j\leq J}\lvert C_j\rvert e^{w_L(\lvert \operatorname{Re} z_j \rvert)} < \infty .\]    
  Choose a complex number \(\alpha_j \in \mathbb{C}\setminus \overline{T_h}\) for each $j = 1, \ldots J,$  such that the line segment \([\alpha_j, z_j]\) lies in \(\mathbb{C}\setminus \overline{T^{F, G}}\). 
    By Lemma \ref{Pole-pushing lemma}, we 
  obtain
     a rational function \(R_j\) with \(\alpha_j\) as only pole that satisfies \[\sup_{\xi \in \overline{T^{F, G}}}\bigg\lvert \frac{1}{z_j-\xi}-R_j(\xi)\bigg\rvert \leq \frac{\varepsilon}{2D  J}.\]
This implies the result.
  \end{claimproof}

Combining the two previous Claims, we find   \(J \in \mathbb{N}\), \(z_1,\ldots,z_J \in \Gamma_1 \cup \Gamma_2\), and   rational functions \(R_1,\ldots,R_J\) with poles in \(\mathbb{C}\setminus \overline{T^{bF, bG}}\)
 and bounded on \(\overline{T^{bF, b G}}\)
such that
\begin{equation}\label{app}
\sup_{\xi\in \overline{T^{F, G}}} e^{w_N(\lvert \operatorname{Re}\xi \rvert)}\bigg\lvert f(\xi) - \sum_{j=1}^J C_j  P(z_j-\xi) R_j(\xi) \bigg\rvert \leq \varepsilon . 
\end{equation}
Now define
\[
g(\xi) = \sum_{j=1}^J C_j  P(z_j-\xi) R_j(\xi) ,
\]
for \(\xi \in T^{bF, bG}\).
Clearly, \(g\) belongs to \(\mathcal{H}(T^{bF,bG})\), satisfies the inequality in  \eqref{equation 4.1}, and, since
\begin{align*}
    \sup_{\xi \in \overline{T^{bF,bG}}} (e^{w_K(\lvert \mathrm{Re \ }\xi)} \lvert g(\xi)\rvert) &\leq C\sup_{\xi \in \overline{T^{bF,bG}}} \sum_{j=1}^J \lvert C_j\rvert e^{w_{\tilde{K}}(\lvert \mathrm{Re \ }\xi - \mathrm{Re \ }z_j\rvert)}e^{w_{\tilde{K}}(\lvert \mathrm{Re \ }z_j\rvert)} \lvert P(z_j-\xi)\rvert \lvert R_j(\xi)\rvert \\& \leq C
    |Q(0)|
    \sum_{j=1}^J \lvert C_j\rvert e^{w_{\tilde{K}}(\lvert \mathrm{Re \ }z_j\rvert)}\sup_{\xi \in \overline{T^{bF,bG}}}\lvert R_j(\xi)\rvert  < \infty ,
\end{align*}
we 
have \(g \in \mathcal{U}_{w_K}(T^{bF,bG})\). 
\end{proof}
\section{Proof of the main result}\label{proof main theorem}

This section is devoted to the proof of Theorem \ref{main theorem}.  We start with an auxiliary result that is inspired by the proof of the well-known fact, due to Grothendieck, that an elliptic constant coefficient partial differential operator  \(P(D) \colon C^{\infty}(\Omega) \rightarrow C^{\infty}(\Omega)\), \(\Omega \subseteq \mathbb{R}^d\) open and \(d\geq 2\), does not have a continuous linear right inverse; see for example \cite[Appendix C]{treves1967convexandePDE}. 
\begin{proposition}\label{non-triviality kernel} Let \(F, G \in \mathcal{F}(\mathbb{R})\) and let \(W = (w_N)_{N\in \mathbb{N}}\) be a weight function system.
 Let \(P(D)\) be an  elliptic constant coefficient partial differential operator  with $\operatorname{deg} P \geq 1$. Suppose that \(P(D) \colon \mathcal{K}_{W}(T^{F,G}) \rightarrow \mathcal{K}_{W}(T^{F,G})\) is surjective. Then, the space \(\operatorname{ker }P(D) \subseteq \mathcal{K}_{W}(T^{F,G})\) is  non-trivial.
\end{proposition}
\begin{proof}
Suppose that \(\operatorname{ker }P(D) \subseteq \mathcal{K}_{W}(T^{F,G})\) is trivial. Then, \(P(D) \colon \mathcal{K}_{W}(T^{F,G}) \to   \mathcal{K}_{W}(T^{F,G})\)  is a continuous linear bijection. By the open mapping theorem, it is then an isomorphism between Fréchet spaces. Hence, this map has a continuous linear inverse \(I \colon \mathcal{K}_{W}(T^{F,G}) \rightarrow \mathcal{K}_{W}(T^{F,G})\). 
  Fix \(a \in (0,1)\). There are \(b \in (0,1)\), \(A >0 \), and \(N \in \mathbb{N}\) such that, for all \(f \in \mathcal{K}_{W}(T^{F,G})\),
  \begin{equation}\label{fundamental estimate}
  \sup_{\xi \in \overline{T^{aF, aG}}} \lvert I(f)(\xi) \rvert \leq A \sup_{\xi \in \overline{T^{bF,bG}}, \lvert \alpha \rvert \leq N} e^{w_N (\lvert \mathrm{Re \ }\xi\rvert)}\lvert f^{(\alpha)}(\xi)\rvert .
  \end{equation}
  Set \(c = \max\{a,b \}\). Next, take \(z \in T^{F,G} \setminus \overline{T^{cF,cG}}\) and \(\varepsilon > 0\) such that \(\overline{B}(z,\varepsilon) \subseteq T^{F,G} \setminus \overline{T^{cF,cG}}\). As usual, set \(\mathcal{D}(B(z,\varepsilon)) = \{ \varphi \in C^\infty(\mathbb{R}^2) \mid \operatorname{supp } \varphi \subseteq B(z,\varepsilon)\} \). We 
 now prove 
  the following claim.
  \begin{claim}\emph{
 \(I(\mathcal{D}(B(z,\varepsilon))) \subseteq \mathcal{D}(B(z,\varepsilon))\). Consequently, \(P(D) \colon \mathcal{D}(B(z,\varepsilon)) \rightarrow \mathcal{D}(B(z,\varepsilon))\) is surjective.}
  \end{claim}
  \begin{claimproof}
 Take an arbitrary \(\varphi \in \mathcal{D}(B(z,\varepsilon))\) and choose \(r \in (0,\varepsilon)\) such that \(\operatorname{supp }\varphi\subseteq B(z,r)\). On the one hand, we have that \(P(D)(I(\varphi))(\xi) = \varphi(\xi) = 0\)  for all \(\xi \in T^{F,G} \setminus \overline{B}(z,r)\). Hence, as $P(D)$ is elliptic, we obtain that \(I(f)\) is real analytic on the open domain \(T^{F,G} \setminus \overline{B}(z,r)\). 
 On the other hand, inequality \eqref{fundamental estimate} implies that \(I(\varphi) \equiv 0\) on \(T^{aF,aG}\). The uniqueness property for real analytic functions yields 
  \(I(\varphi)\equiv 0\) on the open domain \(T^{F,G} \setminus \overline{B}(z,r)\). Hence, \(I(\varphi) \in \mathcal{D}(B(z,\varepsilon))\). 
 \end{claimproof}
 
Since  \(P(D) \colon \mathcal{D}(B(z,\varepsilon)) \rightarrow \mathcal{D}(B(z,\varepsilon))\) is surjective, its transpose \( P(D)^t = P(-D): \\ \mathcal{D}^{\prime}(B(z,\varepsilon)) \rightarrow \mathcal{D}^{\prime}(B(z,\varepsilon))\) is injective. This means that \(\operatorname{ker }P(-D) \subseteq  \mathcal{D}^{\prime}(B(z,\varepsilon))\) is trivial. However, this is a contradiction, since,  due to the fundamental theorem of algebra, there are always \(\lambda, \mu\in \mathbb{C}\) such that \(f(x,y) = e^{\lambda x+\mu y}\) satisfies \(P(-D)f = 0\). 
\end{proof}
We are ready to show Theorem \ref{main theorem}.  
\begin{proof}[Proof of Theorem \ref{main theorem}]
 \((iii)\implies (i)\). We use the abstract Mittag-Leffler lemma (Proposition \ref{Mittag-Leffler}) to show this implication. Let \((a_N)_{N\in\mathbb{N}}\) be an increasing positive sequence that converges to \(1\). We define the projective spectrum \(\mathcal{X}\) as the sequence \((\mathcal{K}_{w_N}(T^{a_NF,a_NG}))_{N \in \mathbb{N}}\) together with the natural restriction maps \(\rho_M^N \colon \mathcal{K}_{w_M}(T^{a_M F,a_M G})  \rightarrow \mathcal{K}_{w_N}(T^{a_N F,a_N G})\). 
  Consider the morphism \(f \colon \mathcal{X} \rightarrow \mathcal{X}\) given by
  \[(\overline{\partial} \colon \mathcal{K}_{w_N}(T^{a_N F, a_N G}) \rightarrow \mathcal{K}_{w_N}(T^{a_N F, a_N G}))_{N\in \mathbb{N}}.\]
  Then, \(\operatorname{Proj } \mathcal{X} = \mathcal{K}_{W}(T^{F,G})\) and \(\operatorname{Proj} f: \operatorname{Proj } \mathcal{X} \to \operatorname{Proj } \mathcal{X}\) coincides with \(\allowdisplaybreaks \overline{\partial} \colon \mathcal{K}_{W}(T^{F,G}) \rightarrow \mathcal{K}_{W}(T^{F,G})\).
Hence, by
Theorems \ref{Mittag-Leffler} and  \ref{density condition}, it suffices to show the following two properties: 
 \begin{enumerate}
     \item[$(AP)$] For all \(N \in \mathbb{N}\) there is \(M > N\) such that for all \(K > M\):
     \[
     \rho_M^N(\mathcal{U}_{w_M}(T^{a_M F, a_M G})) \subseteq \overline{\rho_{K}^N(\mathcal{U}_{w_K}(T^{a_K F, a_K G}))}^{\mathcal{U}_{w_N}(T^{a_N F, a_N G})} .
     \]
     \item[$(LP)$] For all \(N \in \mathbb{N}\) there is \(M > N\) such that for all \(g \in \mathcal{K}_{w_M}(T^{a_M F, a_M G})\) there is \(f \in \mathcal{K}_{w_N}(T^{a_N F, a_N G})\) such that    $\overline{\partial} f = g$ on $T^{a_N F, a_N G}$.
 \end{enumerate}
 
In view of Lemma \ref{lemma-df}, property $(AP)$ follows from Theorem \ref{approximation theorem} with \(a = a_M/a_N\) and \(b = a_K / a_N\).  We now establish property $(LP)$. 
Let \(N \in \mathbb{N}\) be arbitrary.
 Using Lemma \ref{Subadditivity}, 
we find \(M > N\) and \(C > 1\) such that, for all \(t,s \geq 0\),
    \[
    w_N(t+s) \leq w_M(t)+w_M(s) + \log C .
    \]
Let  \(g \in \mathcal{K}_{w_M}(T^{a_M F, a_M G})\) be arbitrary. Condition \((N)\) implies 
 that there is \(K > M\) such that
\begin{equation}
\label{condNN}
    \int_0^{\infty}e^{w_M(t) - w_K(t)}\mathrm{d}t < \infty .
 \end{equation}
    Choose \(h > 2a_M\max\{\sup_{t\in\mathbb{R}} F(t), \sup_{t\in \mathbb{R}}G(t)\}\). By Lemma \ref{Lemma 2}, there is \(Q \in \mathcal{H}(T_h)\) such that  $\lvert Q(\xi)\rvert \geq e^{w_K(\lvert \mathrm{Re \ }\xi\rvert)}$ for all  $\xi \in T_h$. Define  \(P = Q(0)/Q\). Fix \(b \in (a_N , a_M)\). Lemma \ref{Lemma 1} allows us to choose \(\varphi \in C^{\infty}(\mathbb{R}^2)\) with \(\operatorname{ supp }\varphi \subseteq T^{b F, b G}\) such that \(\varphi \equiv 1\) on \(\overline{T^{a_NF, a_NG}}\) and $\sup_{z \in \mathbb{C}} |\varphi^{(\alpha)}(z)| < \infty$ for all $\alpha \in \mathbb{N}^2$. Set  \( \displaystyle f = (g\varphi)*(P(z)/(\pi z)) \in C^\infty(T^{a_NF,a_NG})\). Since $P(0) = 1$, \(P \in \mathcal{H}(T_h)\), \((\pi z)^{-1}\) is a fundamental solution of \(\overline{\partial}\), and \(\varphi \equiv 1\) on \(T^{a_NF, a_NG}\), we have 
     \(\overline{\partial} f = g \) on \(T^{a_NF, a_NG}\). 
    We now show that \(f \in  \mathcal{K}_{w_N}(T^{a_NF,a_NG})\). 
 The relation \eqref{condNN} and the integrability of $z^{-1}$ near $0\in\mathbb{C}$      imply that   
    $$
    \int_{T_h} e^{w_M(\lvert \mathrm{Re \ }z \rvert)}\frac{\lvert P(z)\rvert}{\pi\lvert z \rvert}\mathrm{d}x\mathrm{d}y  < \infty.
    $$
Hence, we have, for all \(\alpha \in \mathbb{N}^2\) and \(\xi \in T^{a_NF, a_NG}\),
\begin{align*}
       &  e^{w_N (\lvert \mathrm{Re \ }\xi\rvert)} \lvert f^{(\alpha)}(\xi)\rvert \\
        & \leq e^{w_N(\lvert \mathrm{Re \ }\xi\rvert)}\sum_{\beta \leq \alpha}\binom{\alpha}{\beta}\   \int_{\xi - T^{bF, bG}} 
        \lvert g^{(\beta)}(\xi-z)\rvert \lvert \varphi^{(\alpha - \beta)}(\xi-z)\rvert \frac{\lvert P(z)\rvert}{\pi \lvert z\rvert}\mathrm{d}x\mathrm{d}y \\
        &\leq C \sum_{\beta \leq \alpha} \binom{\alpha}{\beta} \sup_{w\in \mathbb{C}}\lvert \varphi^{(\alpha - \beta)}(w)\rvert \sup_{w\in \overline{T^{bF, bG}}} (e^{w_M(\lvert \mathrm{Re \ }w\rvert}\lvert g^{(\beta)}(w)\rvert) \int_{T_h} e^{w_M(\lvert \mathrm{Re \ }z \rvert)}\frac{\lvert P(z)\rvert}{\pi\lvert z \rvert}\mathrm{d}x\mathrm{d}y \\
        &< \infty .
    \end{align*}

  \((i)\implies (iv)\).  This follows from Proposition \ref{non-triviality kernel} with $P(D)  = \overline{\partial}$.
 
  \((iv)\implies (iii)\).  We follow the same idea as in the proof of \cite[Proposition 4.2]{debrouwere2018non}. Fix \(a\in (0,1)\) and choose \(h>0\) such that \(h < a\min\{\inf_{t\in\mathbb{R}} F(t), \inf_{t\in \mathbb{R}} G(t)\}\). Applying \cite[Lemma 3.4]{debrouwere2018non}, we get that any \(f \in \mathcal{H}(T_h) \setminus \{0\}\) that is continuous and bounded on \(\overline{T_h}\) must satisfy
    \[
    -\infty < \int_{-\infty}^{\infty} \log\lvert f(t)\rvert e^{-\pi \lvert t \rvert / h}\mathrm{d}t .
    \]
    Choose  \(f \in \mathcal{U}_W(T^{F,G})\setminus \{0\}\). For all \(N \in \mathbb{N}\), 
    \begin{align*}
        -\infty &< \int_{-\infty}^{\infty} \log \lvert f(t)\rvert e^{-\pi \lvert t \rvert / h} \mathrm{d}t \\
        &= \int_{-\infty}^{\infty} \log \lvert e^{w_N(N\lvert t \rvert)}f(t)\rvert e^{-\pi \lvert t \rvert / h}\mathrm{d}t - \int_{-\infty}^{\infty}w_N(N\lvert t \rvert)e^{-\pi \lvert t \rvert / h}\mathrm{d}t \\
        &\leq \log\left( \sup_{\xi \in \overline{T^{aF, aG}}} e^{w_N(N\lvert \mathrm{Re \ }\xi\rvert)} \lvert f(\xi)\rvert \right) \int_{-\infty}^{\infty} e^{-\pi \lvert t \rvert / h}\mathrm{d}t - \frac{1}{N}\int_{-\infty}^{\infty}w_N(\lvert t \rvert) e^{-\pi \lvert t \rvert / (Nh)}\mathrm{d}t
    \end{align*}
    Since $W$ satisfies \((\alpha)\), we obtain 
     that \(\sup_{\xi \in \overline{T^{aF, aG}}}e^{w_N(N \lvert\mathrm{Re \ }\xi\rvert)} \lvert f(\xi)\rvert < \infty\). Therefore, 
    \[
    \int_{-\infty}^{\infty} w_N(\lvert t\rvert) e^{-\pi \lvert t \rvert/(Nh)} \mathrm{d}t = 2\int_0^{\infty} w_N(t) e^{-\pi t / (Nh)}\mathrm{d}t < \infty .
    \] 
   Since \(N \in \mathbb{N}\) was arbitrary, we may conclude that \(W\) satisfies \((\epsilon)_0\).

 \((i)\implies (ii)\).    By complex conjugation, the operator \(\displaystyle 2\partial = \frac{\partial}{\partial x} - i\frac{\partial}{\partial y}\) is surjective on \(\mathcal{K}_{W}(T^{F,G})\) as well. Since \(\Delta = 4\partial \overline{\partial}\), the result follows.
 
  \((ii)\implies (iv)\).   Proposition \ref{non-triviality kernel} with $P(D) = \Delta$ tells us that there is a harmonic function in \(\mathcal{K}_{W}(T^{F,G})\) that is 
 non-identically zero. By passing to its real or imaginary part, we can find a real-valued harmonic function \(u \in \mathcal{K}_{W}(T^{F,G}) \backslash \{0 \}\). Define \(\displaystyle v = \frac{\partial u}{\partial x} - i \frac{\partial u}{\partial y} \in \mathcal{K}_{W}(T^{F,G})\). Then, \    $
2\overline{\partial} v= \Delta u=0.
    $
     Thus \(v \in \mathcal{U}_{W}(T^{F,G})\). Moreover, \(v\) is non-identically zero; it is therefore our sought non-trivial element of  \(\mathcal{U}_{W}(T^{F,G})\). Indeed, if $v$ would identically vanish, then \(\nabla u = 0\) and  \(u\) would be constant. But  $u$ vanishes at infinity, what would lead to $u \equiv 0$,  a contradiction.
 \end{proof}

\end{document}